\documentclass[12pt]{article}

\usepackage{upgreek}

\usepackage{hyperref}
\usepackage{changes}

\usepackage[normalem]{ulem}
\usepackage{bm}
\usepackage{amsthm}
\usepackage{amsmath}
\usepackage{amssymb}
\usepackage{mathrsfs}
\usepackage{mathtools}
\usepackage{graphicx}
\usepackage{tikz}
\usetikzlibrary{arrows}
\usepackage{enumitem}
\usepackage{comment}
\usepackage{cases}
\usepackage[capitalise]{cleveref}
\usepackage{fullpage}
\usepackage{appendix}

\newtheorem{theorem}{Theorem}

\newtheorem{lemma}[theorem]{Lemma}

\newtheorem{corollary}[theorem]{Corollary}

\theoremstyle{definition}
 
\theoremstyle{remark}

\theoremstyle{remark}

\numberwithin{theorem}{section}

\providecommand{\R}{}

\providecommand{\N}{}

\renewcommand{\R}{\mathbb{R}}

\renewcommand{\N}{{\mathbb N}}

\newcommand{\E}{\textsf{\upshape E}}
\newcommand{\prob}{\textsf{\upshape Pr}}

\newcommand{\Gnp}{\mathcal{G}(n,p)}

\begin{document}

\title{Multiset Metric Dimension of Binomial Random Graphs}

\author{
Austin Eide\thanks{Department of Mathematics, Toronto Metropolitan University, Toronto, ON, Canada; e-mail: \texttt{austin.eide@torontomu.ca}}, 
Pawe\l{} Pra\l{}at\thanks{Department of Mathematics, Toronto Metropolitan University, Toronto, ON, Canada; e-mail: \texttt{pralat@torontomu.ca}}
}

\maketitle

\begin{abstract}
For a graph $G = (V,E)$ and a subset $R \subseteq V$, we say that $R$ is \textit{multiset resolving} for $G$ if for every pair of vertices $v,w$, the \textit{multisets} $[d(v,r): r \in R]$ and $[d(w,r):r \in R]$ are distinct, where $d(x,y)$ is the graph distance between vertices $x$ and $y$. The \textit{multiset metric dimension} of $G$ is the size of a smallest set $R \subseteq V$ that is multiset resolving (or $\infty$ if no such set exists). This graph parameter was introduced by Simanjuntak, Siagian, and Vitr\'{i}k in 2017~\cite{simanjuntak2017multiset}, and has since been studied for a variety of graph families. We prove bounds which hold with high probability for the multiset metric dimension of the binomial random graph $\Gnp$ in the regime $d = (n-1)p = \Theta(n^{x})$ for fixed $x \in (0,1)$.
\end{abstract}

\section{Introduction}

\subsection{Background}

For a finite connected graph $G = (V,E)$, a subset $R \subseteq V$ is \textit{metric resolving} if all vertices in the graph can be distinguished by their shortest path distances to $R$. More formally, $R$ is metric resolving if the vectors $\bm{s}_{R}(v) = (d(v,w))_{w \in R}$ are distinct for all $v \in V$, where $d(v,w)$ is the graph distance between $v$ and $w$. We refer to the vector $\bm{s}_{R}(v)$ as the \textit{metric signature} of $v$ with respect to $R$. The \textit{metric dimension} of $G$, denoted $\beta(G)$, is the size of a smallest metric resolving set for $G$. Note that $V$ itself is always metric resolving, so $\beta(G) \leq |V|$ for any graph. We remark that the definition of a metric resolving set can be extended to cover graphs which are not connected by declaring that vertices in distinct components have distance $\infty$; with this definition, the metric dimension of $G$ is simply the sum of the metric dimensions of its components. Graph metric dimension can be thought of as a discrete analogue to \textit{trilateration}, the method employed by Global Positioning Systems (GPS) to determine locations of points on the earth using other landmarks~\cite{tillquist2023getting}.

The study of the metric dimension of graphs dates to the 1970s, when it was first proposed by Slater~\cite{slater1975leaves} and Harary and Melter~\cite{harary1976metric}. Since then it has received significant attention from mathematicians and computer scientists, and the metric dimension is now known for a wide variety of graph families. We refer to the recent survey~\cite{tillquist2023getting} for a fairly comprehensive list of these families, as well as some interesting applications. We make particular note of previous works on the metric dimension of random graphs. In~\cite{bollobas2013metric}, Bollob\'{a}s, Mitsche, and Pra\l{}at investigated the metric dimension of $\Gnp$ when the average degree $d = (n-1)p$ is at least $\log^{5}n$; a recent paper of D\'{i}az, Hartle, and Moore treats the cases $\log n < d < \log^{5}n$~\cite{diaz2025metric}. (Note that $\Gnp$ is disconnected w.h.p.\footnote{An event holds with high probability (w.h.p.) if it holds with probability tending to one as $n \to \infty$.}\ for $d \leq (1-\epsilon)\log n$, where $\epsilon > 0$ is an arbitrarily small constant.) \'{O}dor and Thiran also consider a sequential version of metric dimension on $\Gnp$ in which a player attempts to locate a hidden vertex $v^{*}$ by making sequential queries of the distances $d(v,v^{*})$ for $v \in V$~\cite{odor2021sequential,odor2022role}. In~\cite{dudek2022localization}, Dudek, English, Frieze, MacRury, and Pra\l{}at study a similar process on $\Gnp$ in which the hidden vertex $v^{*}$ is allowed  to move along an edge in each round. See also~\cite{dudek2019note} for weaker bounds in the regime in which $\Gnp$ has diameter two w.h.p. Lichev, Mitsche, and Pra\l{}at additionally study this process on random geometric graphs~\cite{lichev2023localization}. See also~\cite{mitsche2015limiting} for the metric dimension of random trees and very-sparse $(d<1)$ random graphs and~\cite{tillquist2019multilateration} for related results on the stochastic block model.

In this paper we study a coarser version of the metric dimension called \textit{multiset metric dimsension}. The idea is that vertices should be distinguished by their multisets of distances to a set $R$, rather than by their pointwise distances to $R$. To define the parameter formally, we make a few definitions. For a vertex $v \in V$ and integer $k \geq 0$, define
$$
	S_{k}(v) := \{w \in V \,:\,d(v,w) = k\} \quad \text{ and } \quad N_{k}(v) := \{w \in V\,:\,d(v,w) \leq k\} = \bigcup_{j=0}^{k}S_{j}(v).
$$
These are, respectively, the sphere and ball of radius $k$ around $v$ in $G$. Similarly, for a subset $V' \subseteq V$ and $k \geq 0$, define
$$
	S_{k}(V') := \bigcup_{v \in V'}S_{k}(v) \quad \text{ and }\quad N_{k}(V') := \bigcup_{v \in V'}N_{k}(v).
$$
Finally, for $V', R \subseteq V$, define 
$$
	S_{k}^{R}(V') := S_{k}(V') \cap R \quad\text{ and }\quad N_{k}^{R}(V') := N_{k}(V') \cap R.
$$ 

Now, define the \textit{multiset signature} of $v$ with respect to $R \subseteq V$ as the vector 
$$
	\bm{m}_{R}(v) := (|S_{k}^{R}(v)|)_{k=0}^{\text{diam}(G)} \in \N_{0}^{\text{diam}(G) + 1},
$$
where $\text{diam}(G) = \max_{u, v \in V} d(u,v)$ is the diameter of graph $G$ and $\N_{0}$ is the set of nonnegative integers. Observe that $\bm{m}_{R}(v)$ is a representation of the metric signature $\bm{s}_{R}(v)$ when $\bm{s}_{R}(v)$ is considered as a multiset rather than a vector. A set $R$ is \textit{multiset resolving} if $\bm{m}_{R}(v) \neq \bm{m}_{R}(w)$ for all $v \neq w$, and we define the multiset metric dimension $\beta_{\text{ms}}(G)$ to be the size of a smallest multiset resolving set for $G$ if such a set exists, and define $\beta_{\text{ms}}(G) = \infty$ otherwise. This parameter was introduced by Simanjuntak, Siagian, and Vitr\'{i}k in 2017~\cite{simanjuntak2017multiset}, and has since been studied by multiple authors---see, for instance,~\cite{alfarisi2020note},~\cite{bong2021some}, and~\cite{hakanen2024complexity}. As with metric resolving sets, the definition of a multiset resolving set can be extended to apply to graphs which are not connected.

A connected graph $G$ need not have finite multiset metric dimension. For instance, suppose $G$ has diameter at most $2$ and $|V| \geq 4$. It is easy to see that in order for a set $R \subseteq V$ to be multiset resolving, every vertex in $R$ must have a distinct number of neighbors in $R$. Since each vertex in $R$ has between $0$ and $|R|-1$ neighbors in $R$, there must be exactly one vertex in $R$ with $i$ neighbors in $R$ for each $i \in \{ 0,1,\dots,|R|-1 \}$. But this is impossible unless $|R|-1 = 0$, since $R$ cannot contain both a vertex with $|R|-1$ neighbors in $R$ and a vertex with $0$ neighbors in $R$ if $|R|>1$. So we must have $|R| = 1$. But by a similar argument to the above, in order for a set $R$ to distinguish vertices in $V \setminus R$, we must have $|R| \geq \lceil (|V|-1)/2 \rceil$, which is at least $2$ for $|V| \geq 4$. (This result was originally established in~\cite[Theorem~3.1]{simanjuntak2017multiset}.)

To remediate this issue, Gil-Pons et.\ al.\ in~\cite{gil2019distance} define the \textit{outer multiset metric dimension} $\beta_{\text{ms}}^{\text{out}}(G)$ as the size of a smallest $R$ such that $\bm{m}_{R}(v) \neq \bm{m}_{R}(w)$ for all $v \neq w$ with $v,w \in V \setminus R.$ Trivially, any set $R$ of size $|V|-1$ is outer-multiset resolving, and hence always $\beta_{\text{ms}}^{\text{out}}(G) \leq |V|-1$; in fact, it is shown in~\cite{klavvzar2023further} that $\beta_{\text{ms}}^{\text{out}}(G) = |V| - 1$ if and only if $G$ is regular and has diameter at most $2$. Clearly we have $\beta_{\text{ms}}(G) \geq \beta_{\text{ms}}^{\text{out}}(G)$ for any graph $G$. On the other hand, we claim that $\beta_{\text{ms}}^{\text{out}}(G) \geq \beta(G)$. Indeed, if $R$ is outer-multiset resolving, then $\bm{s}_{R}(v) \neq \bm{s}_{R}(w)$ for all $v\neq w$ with $v,w \in V \setminus R$; but it is easy to see that this is sufficient to ensure that $\bm{s}_{R}(v) \neq \bm{s}_{R}(w)$ for \textit{all} $v \neq w$, so $R$ is also metric resolving. Hence,
$$
    \beta_{\text{ms}}(G) \geq \beta_{\text{ms}}^{\text{out}}(G) \geq \beta(G).
$$

Multiset resolving sets inherit some interesting applications from metric resolving sets. One example is source localization~\cite[Section 7.1]{tillquist2023getting}. Suppose that at time $t = 0$ a single (unknown) vertex $v_{0}$ of a connected graph $G$ is infected, and that any uninfected vertex which is adjacent to an infected vertex at time $t-1$ becomes infected at time $t$, for each $t =1,2,\dots$. We would like to identify the source vertex $v_{0}$ by observing the spread of the infection. One strategy for achieving this is to place a sensor at each vertex of some multiset resolving set $R$ for $G$, with each sensor emitting a signal at the first moment its vertex becomes infected. For $t = 0,1,2,\dots$, let $N_{t}$ be the number of signals emitted by vertices in $R$ at time $t$. (Note that $N_{t}$ is simply the number of vertices in $R$ at distance $t$ from $v_{0}$.) Then since $R$ is multiset-resolving for $G$, the identity of $v_{0}$ is completely determined by the sequence $\{N_{t}\}_{t=0}^{\text{diam}(G)}$. 

Metric resolving sets have also been shown to be useful as a tool for certain instances of \textit{graph embedding}, where the goal is map the vertices of $G$ to real $N$-dimensional space while preserving structural properties of the graph. In applications, it is preferable for $N$ to be small. For example, in~\cite{tillquist2019low} the authors use small resolving sets to find low-dimensional embeddings of graphical representations of genomic sequences, allowing for the application of a variety of clustering algorithms. The embedding is obtained by mapping each vertex to its metric signature with respect to some small metric resolving set; thus, given a metric resolving set $R$, the embedding maps the graph into $\R^{|R|}$. Importantly, for the application considered in~\cite{tillquist2019low}, the metric resolving set embedding is nearly an isometry, i.e., the Euclidean distance between the images of any pair of vertices is quite close to the graph distance between that pair. 

Multiset signatures with respect to a multiset resolving set also provide a natural graph embedding, and we propose that this embedding could prove useful for some applications. While $\beta_{\text{ms}}(G) \geq \beta(G)$ for any graph $G$, we observe that the dimension of the \textit{embedding} produced by a multiset resolving set for $G$ has dimension $\text{diam}(G) + 1$, as multiset signatures are vectors of this length. (Metric resolving set embeddings have dimension at least $\beta(G)$.) Our results demonstrate that for random graphs on $n$ vertices with average degree on the order of $n^{x}$ for $x > 0$ sufficiently small, a multiset resolving set exists w.h.p.; since the diameter of this graph is $O\left( \frac{1}{x} \right)$ w.h.p.\ (see \cite{bollobas1998random}), we get an embedding into $\R^{O\left(\frac{1}{x}\right)}$. On the other hand, results from \cite{bollobas2013metric} show that corresponding metric resolving set embedding has dimension which is logarithmic in $n$ for the ``best case" values of $x$ and polynomial in $n$ for the ``worst case." The extent to which either embedding preserves graph distances is an interesting question for future research.

\subsection{Definitions}
The binomial random graph $\Gnp$ is formally defined as a distribution over the class of graphs with the set of vertices $[n]:=\{1,\ldots,n\}$ in which every pair $\{i,j\} \in \binom{[n]}{2}$ appears independently as an edge in $G$ with probability~$p$. Note that $p=p(n)$ may (and usually does) tend to zero as $n$ tends to infinity. Most results in this area are asymptotic by nature. For more about this model see, for example,~\cite{bollobas1998random,frieze2016introduction,JLR}. 

We let $d = (n-1)p$, where $n$ is the number of vertices and $p$ is the edge probability of $\Gnp$. Given $d$, we reserve the index $i^{*}$ to denote the largest integer $i = i(n)$ such that $d^{i} = o(n)$ as $n \to \infty$. We let $c = c(n) = \frac{d^{i^{*}+1}}{n}$. (These parameters essentially determine the diameter of $\Gnp$.) To make our results simpler to state, we assume throughout that $d = n^{x+O(\log^{-1}n)}$ for some fixed $x \in (0,1)$. (Note that this is equivalent to $d = \Theta(n^{x})$, though in the proofs it is slightly easier to work with the form $n^{x + O(\log^{-1}n)}$.) With this simplifying assumption on $d$, we note that $i^{*} = \lfloor 1/x \rfloor$, unless $x = 1/k$ for some $k \geq 2$, in which case $i^{*} = k-1$. In particular, $i^{*}$ is a function of $x$ alone (and not $n$).

\subsection{Concentration tools}
Let $X \in \textrm{Bin}(n,p)$ be a random variable distributed according to a Binomial distribution with parameters $n$ and $p$. We will use the following consequences of \emph{Chernoff's bound} (see e.g.~\cite[Theorem~2.1]{JLR}): for any $t \ge 0$ we have
	\begin{eqnarray}
		\prob( X \ge \E X + t ) &\le& \exp \left( - \frac {t^2}{2 (\E X + t/3)} \right)  \label{chern1} \\
		\prob( X \le \E X - t ) &\le& \exp \left( - \frac {t^2}{2 \E X} \right).\label{chern2}
	\end{eqnarray}
Note that if $t \leq \frac{3 \E X}{2}$, the right-hand side of (\ref{chern1}) is at most $\exp \left( - \frac {t^2}{3 \E X } \right)$; since this also upper bounds (\ref{chern2}), we have the following corollary: for any $0 \leq t \leq \frac{3 \E X}{2}$, 

	\begin{equation}\label{chern3}
		\prob(|X - \E X| \geq t ) \leq 2 \exp \left(  - \frac {t^2}{3 \E X } \right).
	\end{equation}

\subsection{The function $f_{x}(y)$}
For $x \in (0,1]$, define the function $f_{x}$ on $[0,1]$ by
$$
	f_{x}(y) := \sum_{i=0}^{\left\lfloor \frac{1}{x} \right\rfloor}\max\{ix +y- 1,0\}.
$$
The function $f_{x}$ arises naturally in the study of the lower and upper bounds for $\beta_{\text{ms}}(\Gnp)$, and we motivate its appearance in the context of each bound at the beginnings of their respective sections (Sections \ref{sec:upper_bound} and \ref{sec:lower_bound}, respectively.) Here, we remark on some basic properties of $f_{x}$.

\begin{lemma}\label{lem:exponent_func1}
	For $x \in (0,1]$, the function $f_{x}$ is continuous, identically $0$ on $\left[0, 1-\left\lfloor\frac{1}{x}\right\rfloor x\right]$, and strictly increasing on $\left[1-\left\lfloor\frac{1}{x}\right\rfloor x, 1\right].$
\end{lemma}

The proof of Lemma \ref{lem:exponent_func1} is elementary and follows immediately from the definition of the function $f_{x}(y)$, and so we omit it.

\begin{lemma}\label{lem:exponent_func2}
	Let $x \in (0,1]$ and $y \in [0,1]$ be such that $f_{x}(y) > 0$ (or, equivalently, that $1-\left\lfloor\frac{1}{x}\right\rfloor x < y \leq 1$). Then for any $\epsilon > 0$, we have
	$$
		f_{x}(y+\epsilon) - f_{x}(y) \geq \epsilon.
	$$
\end{lemma}

\begin{proof} 
	Since we assume that $f_{x}(y) > 0$, it must be the case that the largest summand in $f_{x}(y)$---namely, $\max\left\{\left\lfloor\frac{1}{x}\right\rfloor x + y - 1,0\right\}$---is positive. Increasing $y$ by $\epsilon$ increases this summand by $\epsilon$; none of the other summands decrease after the change in $y$, so the result follows.
\end{proof}

\begin{lemma}\label{lem:y1_and_y4}
    If $x \in (0,1/8]$, then there is a unique solution $y_{4} = y_{4}(x) \in (0,1)$ to $f_{x}(y) = 4.$ Similarly, if $x \in (0,1/2]$, then there is a unique solution $y_{1} = y_{1}(x) \in (0,1)$ to $f_{x}(y) = 1.$
\end{lemma}

\begin{proof}
    We observe that, since the function $f_{x}$ is continuous, identically $0$ on $\left[0, 1-\left\lfloor\frac{1}{x}\right\rfloor x \right]$, and strictly increasing on $\left[1-\left\lfloor\frac{1}{x}\right\rfloor x, 1\right]$ by Lemma \ref{lem:exponent_func1}, as long as $f_{x}(1) > 4$, there is necessarily a unique solution to $f_{x}(y) = 4$ in $(0,1)$. We have
	$$
		f_{x}(1) = \sum_{i=0}^{\lfloor1/x\rfloor}ix = \frac{\left(\left\lfloor\frac{1}{x}\right\rfloor + 1\right)\left\lfloor\frac{1}{x}\right\rfloor \cdot x}{2}.
	$$
If $x \leq \frac{1}{8}$, then
	$$
		 \frac{\left(\left\lfloor\frac{1}{x}\right\rfloor + 1\right)\left\lfloor\frac{1}{x}\right\rfloor \cdot x}{2} > \frac{\left\lfloor\frac{1}{x}\right\rfloor}{2} \geq \frac{8}{2} = 4
	$$
where we use $(\lfloor1/x\rfloor + 1)x > 1$ for $x > 0$. 

The existence of a unique solution $y_{1}$ to $f_{x}(y) = 1$ for $x \in (0,1/2]$ can be established using the same argument, since $f_x(1) > 1$ for any $x \le 1/2$.
\end{proof}

\subsection{Main result}

Our main result is the following.

\begin{theorem}\label{thm:main_result}
	If $x \in \left(0, \frac{1}{8}\right]$ and $d = (n-1)p = n^{x + O(\log^{-1}n)}$, then w.h.p.
		$$
			\beta_{\textup{ms}}(\Gnp) \leq n^{y_{4} + O(\log^{-1}n)},
		$$
	where $y_{4} = y_{4}(x) \in (0,1)$ is the unique solution to $f_{x}(y) = 4$. On the other hand, if $x \in \left(0, \frac{1}{2} \right]$ and $d = n^{x + O(\log^{-1}n)}$, then w.h.p.
		$$
			\beta_{\textup{ms}}(\Gnp) \geq n^{y_{1} + O(\log^{-1}n)},
		$$
	where $y_{1} = y_{1}(x)$ is the unique solution to $f_{x}(y) = 1$. Finally, if $x > \frac{1}{2}$, then w.h.p. 
		$$
			\beta_{\textup{ms}}(\Gnp) = \infty.
		$$
\end{theorem}

We remark that \cref{thm:main_result} is equivalent to saying that if $d = \Theta(n^{x})$ for some fixed $x \in (0,1/8],$ then $\beta_{\text{ms}}(\Gnp) = O(n^{y_{4}(x)})$ w.h.p., and if $d = \Theta(n^{x})$ for fixed $x \in (0,1/2]$, then $\beta_{\text{ms}}(\Gnp) = \Omega(n^{y_{1}(x)})$ w.h.p. (Recall from \cref{lem:y1_and_y4} that the equation $f_{x}(y) = 4$ has a unique solution if $x \in (0, 1/8]$, and $f_{x}(y) = 1$ has a unique solution if $x \in (0,1/2].$) The gap between the lower and upper bounds is still quite large; indeed, for the range of $x$ for which both  bounds apply, the ratio of the upper bound to the lower goes to infinity polynomially in $n$. It is also still uncertain whether $\beta_{\text{ms}}(\Gnp)$ is even finite w.h.p.\ for $x \in \left(\frac{1}{8},\frac{1}{2} \right]$. Recall that we show in the introduction that $\beta_{\text{ms}}(G) = \infty$ for all $G$ of diameter at most $2$ on at least $4$ vertices. Thus the result for $x > \frac{1}{2}$ follows immediately from the fact that $\Gnp$ has diameter at most $2$ in the regime $d = n^{x}$ for $x>\frac{1}{2}$~\cite{bollobas1998random}.

In Figure~\ref{fig:y4y1} we plot the functions $y_{1}(x)$ and $y_{4}(x)$. Each curve exhibits the type of ``zig-zag" behaviour also observed for the metric dimension of $\Gnp$ in~\cite{bollobas1998random}, with jumps occurring at reciprocals of integers $\frac{1}{k}$ for $k \geq 1.$ Note that these are precisely the points at which the diameter of $\Gnp$ decreases from $k+1$ to $k$. While obtaining a general formula for $y_{1}(1/k)$ and $y_{4}(1/k)$ appears difficult, with an ad hoc approach they are manageable to compute for specific values of $k$. For instance, we found that $y_{1}(1/2) = 3/4$, $y_{1}(1/3) = 2/3$, $y_{1}(1/4) = 7/12$, and so on.

\begin{figure}[h]
	\centering
		\includegraphics[scale=0.75]{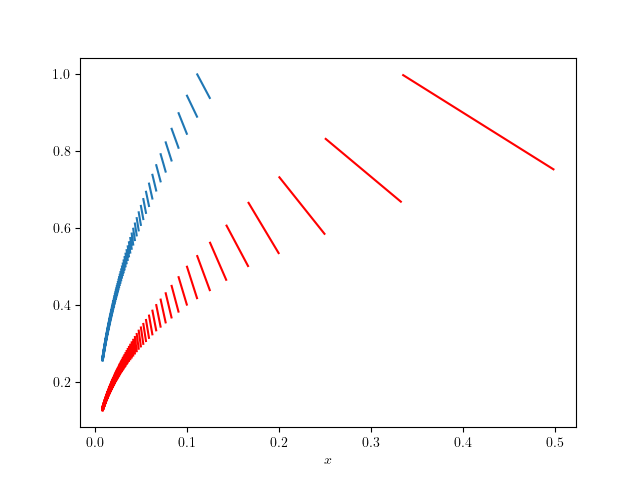}
	\caption{Plots of the points $(x,y_{4})$ satisfying $f_{x}(y_{4}) = 4$ (blue) for $0 < x \leq \frac{1}{8}$ and $(x, y_{1})$ satisfying $f_{x}(y_{1}) = 1$ (red) for $0 < x \leq \frac{1}{2}$.}\label{fig:y4y1}
\end{figure}

\section{Expansion properties}

Here we outline some useful expansion properties of $\Gnp$ which hold w.h.p.\ in the regime we are interested in. We are particularly concerned with the sizes of the spheres $S_{i}(v)$ and $S_i( \{u,v\} )$ for vertices $u$, $v$, and distances $i \geq 0$. Let $d = (n-1)p$ satisfy $d = \omega(\log n)$ and $d = o(n)$. Recall that we let $i^{*} = i^{*}(n)$ be the largest integer $i$ such that $d^{i} = o(n)$ and $c = c(n) = \frac{d^{i^{*}+1}}{n}.$ 

To understand the expansion profile of $\Gnp$, consider revealing the spheres $S_{i}(v)$ centered at a fixed vertex $v$ sequentially. We first reveal the neighbors of $v$, i.e., $S_{1}(v)$. When $d = \omega(\log n)$ we get $|S_{1}(v)| = \deg(v) = (1+o(1))d$. To reveal $S_{2}(v)$, we reveal the edges from $S_{1}(v)$ to $[n] \setminus N_{1}(v)$ conditioned on $S_{1}(v)$ and $|S_{1}(v)| = (1+o(1))d$. We expect to hit roughly $d|S_{1}(v)| = (1+o(1))d^{2}$ new vertices in this step as long as $d^{2} = o(n)$; again we get concentration at $|S_{2}(v)| = (1+o(1))d^{2}$ w.h.p.\ when $d = \omega(\log n)$. The process continues in this way until we reach level $i^{*}$, where we find $|S_{i^{*}}(v)| = (1+o(1))d^{i^{*}}$ w.h.p.

Now consider revealing $S_{i^{*}+1}(v)$ conditioned on levels $1,2,\dots,i^{*}$. The probability that a vertex in $[n] \setminus N_{i^{*}}(v)$ is adjacent to some vertex in $S_{i^{*}+1}(v)$ is 
    $$
        1 - \left(1 - \frac{d}{n}\right)^{|S_{i^{*}}(v)|} = 1 - \left(1 - \frac{d}{n} \right)^{(1+o(1))d^{i^{*}}} = 1 - e^{-(1+o(1))c}
    $$
and thus we expect $n(1-e^{-(1+o(1))c})$ vertices in $S_{i^{*}+1}(v)$. Again we get concentration w.h.p.\ when $d = \omega(\log n)$. Note that it is possible that $c = \Theta(1)$---indeed, this is the case when $d = \Theta(n^{1/k})$ for $k \in \N$.
One can reveal the spheres around pairs of vertices $\{v,w\}$ in a similar manner to the above.

We now make a couple of convenient definitions. Let $\gamma = \max\left\{\sqrt{\frac{\log n}{d}}, \frac{d^{i^{*}}}{n}\right\} =o(1)$. For $V' \subseteq V$ and $i  \in \{ 0,1,2,\dots,i^{*} \} $, define the events
$$
	\mathcal{V}_{i}(V') = \left\{ \, |S_{i}(V')| = (1+O(\gamma))|V'|d^{i} \, \right\}.
$$
and 
$$
	\mathcal{V}_{i^{*}+1}(V') = \left\{\,|S_{i}(V')| = \left(1 - e^{-|V'|c} - \frac{|V'|d^{i^{*}}}{n} + O\left(\gamma + \frac{\log n}{\sqrt{n}} \right) \right)n  \,\right\}.
$$
Let
$$
	\mathcal{V} := \bigcap_{V':|V|' \in \{1,2\}}\bigcap_{i=0}^{i^{*}+1} \mathcal{V}_{i}(V').
$$
All of our results rely on the following lemma. We note that while our main results on the multiset metric dimension only apply to random graphs with average degree on the order of $n^{x}$ for fixed $x >0$, the lemma establishes typical expansion properties of random graphs down to degree $d = \omega(\log n)$.
\begin{lemma}\label{lem:typical_expansion}
	Suppose $d = (n-1)p$ satisfies $d = \omega(\log n)$ and $d = o(n)$. Then the event $\mathcal{V}$ holds w.h.p.\ in $\Gnp.$ 
\end{lemma}

The fact that $\mathcal{V}_{i}(\{v,w\})$ holds w.h.p.\ for any every of vertices $\{v,w\}$ and every $i \in \{ 0,1,\dots,i^{*} \}$ says that $|S_{i}(\{v,w\}| = (1+o(1))(|S_{i}(v)| + |S_{i}(w)|)$, i.e., $S_{i}(v)$ and $S_{i}(w)$ are almost disjoint. Nontrivial intersections can occur at level $i^{*}+1$. Here, \cref{lem:typical_expansion} gives that $|S_{i^{*}+1}(\{v,w\})| = (1+o(1))n(1-e^{-2c})$, while $|S_{i^{*}+1}(v)|$ and $|S_{i^{*}+1}(w)|$ are both $(1+o(1))n(1-e^{-c})$. This implies that $|S_{i^{*}+1}(v) \cap S_{i^{*}+1}(w)| = (1+o(1))n(1-e^{-c})^{2}$. We exploit these facts in the proof of the upper bound of \cref{thm:main_result}.

A version of \cref{lem:typical_expansion} was originally shown to hold for $d \geq (\log n)^{5}(\log\log n)^{2}$ and $d= o(n)$ in \cite{bollobas2013metric}; it was subsequently shown in \cite{odor2021sequential} to hold with the hypothesis $d = \omega(\log n)$ and $d = o(n)$. We refer the reader to those sources for the proof. In particular, for levels $i=0,1,\dots,i^{*}$ see \cite[Lemma 5.1]{odor2021sequential}, and for level $i = i^{*}+1$ see \cite[Lemma 5.3]{odor2021sequential} (a mild extension is required for the latter case). We summarize some useful consequences of \cref{lem:typical_expansion} in the following corollary, which is easily deduced from the lemma. 

\begin{corollary}\label{cor:expansion_cor}
Let $d = (n-1)p$ satisfy $d = \omega(\log n)$ and $d = o(n)$. Then the following hold w.h.p.\ for $\Gnp$: 
	\begin{enumerate}[label = (\roman*)]
		\item For any vertex $v$ and any $i \in \{ 0,1,\dots,i^{*} \} $, $|N_{i}(v)| = (1+o(1))d^{i}$.
		\item For any pair of vertices $v$ and $w$ and any $i  \in \{ 0,1,\dots,i^{*} \} $, $|S_{i}(w) \setminus N_{i}(v)| = (1+o(1))d^{i}$.
		\item If $c = \Theta(1)$, for any vertex $v$, $|N_{i^{*}+1}(v)| = (1+o(1))(1-e^{-c})n$.
		\item If $c = \Theta(1)$, for any pair of vertices $v$ and $w$, $|S_{i^{*}+1}(v) \setminus N_{i^{*}}(w)| = (1+o(1))(1-e^{-c})n$ and $|S_{i^{*}+1}(w) \setminus N_{i^{*}+1}(v)| = (1+o(1))e^{-c}(1-e^{-c})n$.
	\end{enumerate}
\end{corollary}

\section{Upper bound}\label{sec:upper_bound}
Throughout this section, we will condition on $\Gnp$ having the typical expansion property $\mathcal{V}$, so that in particular all statements in \cref{cor:expansion_cor} hold. To prove the upper bound, we use a standard application of the probabilistic method. Our strategy is to show that a random subset of vertices $R$ chosen by including each vertex independently with probability $r/n$ has the following property when $r = r(n)$ is sufficiently large with respect to $d = (n-1)p$: for any pair of vertices $v,w$, 
\begin{equation}\label{eq:distinguish_prob}
	\prob(R\text{ fails to distinguish }v\text{ and }w\,|\,\mathcal{V}) \leq \frac{1}{2n^{2}}.
\end{equation}
Then by a union bound, the probability that $R$ fails to distinguish some pair is at most $\binom{n}{2}\cdot\frac{1}{2n^{2}} \leq \frac{1}{4}$. By Markov's inequality, the probability that $R$ is at least $2\E[|R|] = 2r$ is at most $\frac{1}{2}$. Thus the probability that $R$ is a multiset resolving set and $|R| \leq 2r$ is at least $\frac{1}{4}$, allowing us to conclude that $\Gnp$ has a mutliset resolving set of size at most $2r$ conditional on $\mathcal{V}$. (In practice, $r$ will be $o(n)$, so the bound is nontrivial.) Since $\mathcal{V}$ holds w.h.p., we get $\beta_{\text{ms}}(\Gnp) \leq 2r$ w.h.p.

In the proof, we show that if $d = n^{x + O(\log^{-1}n)}$ for $x \in (0,1)$, the bound \eqref{eq:distinguish_prob} holds when the parameter $r$ is of the order $n^{y}$, where $y$ satisfies $f_{x}(y) \geq 4$. A heuristic explanation is as follows. Recall that $i^{*}$ is the largest $i$ so that $d^{i} = o(n)$. Note first that if $R$ fails to distinguish $v$ and $w$, then $|S_{i}^{R}(v) \setminus S_{i}(w)| = |S_{i}^{R}(w) \setminus S_{i}(v)|$ for all $i$, since each vertex in $S_{i}^{R}(\{v,w\})$ contributes equally to $|S_{i}^{R}(v)|$ and $|S_{i}^{R}(w)|$. For each $i \in \{ 0,1,\dots,i^{*} \} $, $|S_{i}^{R}(v) \setminus S_{i}(w)|$ and $|S_{i}^{R}(w) \setminus S_{i}(v)|$ are independent binomial random variables with means $(1+o(1))\mu_{i}$, where we define $\mu_{i} = \frac{d^{i}r}{n} = \Theta(n^{ix+y -1})$ when $r = n^{y}$. We can then bound the probability that $|S_{i}^{R}(v) \setminus S_{i}(w)| = |S_{i}^{R}(w) \setminus S_{i}(v)|$ by $\min\{1,\,O(\mu_{i}^{-1/2})\} = O\left(n^{-\frac{1}{2}\max\{ix+y-1,0\}} \right)$. (This bound corresponds to the maximum of the probability mass function of the $\text{Bin}(d^{i}, r/n)$ random variable, see \cref{lem:binomial}.) If all of the variables $|S_{i}^{R}(v) \setminus S_{i}(w)|$ and $|S_{i}^{R}(w)\setminus S_{i}(v)|$ were independent for $i \in \{ 0,1,\dots,i^{*} \} $, we could bound the probability that $R$ fails to distinguish $v$ and $w$ by 
\begin{equation}\label{eq:sketch_bound}
	O\left(\prod_{i=0}^{i^{*}}n^{-\frac{1}{2}\max\{ix+y-1,0\}} \right) = n^{-\frac{1}{2}\sum_{i=0}^{i^{*}}\max\{ix+y-1,0\} + O(\log^{-1}n)}.
\end{equation}
In the event, the independence assumption does not hold, since $(S_{i}(v) \setminus S_{i}(w)) \cap (S_{j}(w) \setminus S_{j}(v))$ can be nonempty for $i \neq j$. Even so, by taking a bit more care in the proof we can show that a bound like \eqref{eq:sketch_bound} still holds. When $\frac{1}{k+1} < x < \frac{1}{k}$ for some positive integer $k$, we have $i^{*} = k = \lfloor 1/x \rfloor$, and thus the righthand side of \eqref{eq:sketch_bound} is precisely $n^{-\frac{1}{2}f_{x}(y) + O(\log^{-1}n)}$, suggesting that $y$ must be large enough so that $f_{x}(y) \geq 4$ in order for \eqref{eq:distinguish_prob} to hold. An extension of this argument shows that the same situation holds when $x = \frac{1}{k}$. \cref{lem:y1_and_y4} gives that we may choose such a $y$ provided that $0  < x \leq 1/8$. 

The main result on the upper bound can now be properly stated.

\begin{theorem}\label{thm:upper_bound} Let $x \in \left(0, \frac{1}{8} \right]$ be fixed and let $d = (n-1)p = n^{x + O(\log^{-1}n)}$. W.h.p., $\Gnp$ has a multiset resolving set of size at most $n^{y_{4} +O(\log^{-1}n)}$, where $y_{4} = y_{4}(x) \in (0,1)$ is the unique solution to $f_{x}(y) = 4$.
\end{theorem}

In the proof we also use the following technical result, which bounds the maximum of the probability mass function of a binomial random variable. 

\begin{lemma}\label{lem:binomial}
	Let $Z$ be a binomial random variable with mean $\mu$. Then 
	$$
		\max_{z}\prob(Z = z) \leq \min\left\{1, O\left(\frac{1}{\sqrt{\mu}} \right)\right\}.
	$$
\end{lemma}

We omit the proof of \cref{lem:binomial}, as it follows without too much difficulty from the observation that the maximum of the probability mass function is attained at a $z$ satisfying $z = \mu + O(1)$ and application of Stirling's formula (similar to the proof of the De Moivre-Laplace Theorem, see e.g., \cite[VII.3, Theorem 1]{feller1957introduction}). Now we are ready to prove \cref{thm:upper_bound}.

\begin{proof}[Proof of \cref{thm:upper_bound}]
We condition on $\Gnp$ satisfying $\mathcal{V}$. Let $\epsilon = \frac{C}{\log n}$ with $C = C(x) >0$ a large constant to be determined later, let $y = y(n) = y_{4} + \epsilon$, and let $r = r(n) = n^{y} $. Let $R$ be a random subset of $V$ obtained by including each vertex $v$ in $R$ independently with probability $r/n$. Our goal is to show that for each pair $v,w \in V$, 
\begin{equation}\label{eq:to_show}
	\prob(R \text{ fails to distinguish }v\text{ and }w\,|\,\mathcal{V}) = n^{-\frac{1}{2}f_{x}(y) + O(\log^{-1}n)}.
\end{equation}
By Lemma~\ref{lem:exponent_func2} we have $f_{x}(y) \geq f_{x}(y_{4}) + \epsilon = 4 + \frac{C}{\log n}$. Thus we can choose $C$ sufficiently large so that the right-hand side of \eqref{eq:to_show} is at most $\frac{1}{2n^{2}}$, and hence by a union bound we have
$$
	\prob(R\text{ fails to distinguish some pair of vertices}\,|\,\mathcal{V}) \leq \binom{n}{2} \cdot \frac{1}{2n^{2}} \le \frac{1}{4}.
$$
By Markov's inequality the probability that $|R|$ exceeds $2\E[R] = 2r$ is at most $\frac{1}{2}$. Using this and the above, the probability that $R$ distinguishes all vertices and has size at most $2r = n^{y_{4} + O(\log^{-1}n)}$ is at least $\frac{1}{4}$. We may then conclude that any graph satisfying $\mathcal{V}$ has a multiset resolving set of size at most $n^{y_{4} +O(\log^{-1}n)}$, and thus $\Gnp$ has a multiset resolving set of size at most $n^{y_{4} + O(\log^{-1}n)}$ w.h.p.

Now we focus on showing \eqref{eq:to_show}. Fix $v$ and $w$. We will reveal the set $R$ in stages with respect to these two vertices. Initially none of $R$ is revealed. In stage $0$ we reveal $R \cap N_{0}(\{v,w\}) = \{v,w\}$. In stage $1$ we reveal $R \cap (N_{1}(\{v,w\}) \setminus N_{0}(\{v,w\}))$, and in general in stage $i$ we reveal $R \cap (N_{i}(\{v,w\}) \setminus N_{i-1}(\{v,w\}))$. We reveal $R$ in this way through stage $\lfloor 1/x \rfloor$, though we need to take some care to handle the final stage when $x = \frac{1}{k}$ for some integer $k \ge 8$, as in this case we have $i^{*} = \lfloor1/x\rfloor -1$. We start with the analysis of stages $0$ through $i^{*}$, which applies to all $x \in (0,1/8]$. 

In stage $0$, we simply reveal whether $v$ and $w$ are in $R$. We condition on neither vertex being in $r$, as this occurs with probability $1 - O(r/n) = 1-o(1)$. For $i \geq 1$, condition on the outcomes of stages $0,1,\dots,i-1$. To process stage $i$, we first reveal which vertices in $S_{i}(v) \setminus N_{i-1}(w)$ are in $R$ and also condition on these. At this point we have the following information:
	\begin{enumerate}
		\item $|S_{i}^{R}(v)|$ is deterministic (i.e., it is a function of the information revealed so far);
		\item $|S_{i}^{R}(w)|$ can be expressed as $z_{i} + Z_{i}$, where $z_{i} = |S_{i}^{R}(w) \cap N_{i}(v)|$ is deterministic and $Z_{i} = |S_{i}^{R}(w) \setminus N_{i}(v)|$ is a random variable; 
		\item $Z_{i}$ is a $\text{Bin}(s_{i}, r/n)$ random variable with $s_{i} = |S_{i}(w) \setminus N_{i}(v)|$.
	\end{enumerate}
By the event $\mathcal{V}$, we have $s_{i} = (1+o(1))d^{i} = (1+o(1))n^{ix + O(\log^{-1}n)} = \Theta(n^{ix}).$ Observe that $|S_{i}^{R}(v)|=|S_{i}^{R}(w)|$ if and only if $Z_{i} = |S_{i}^{R}(v)| - z_{i}$. Conditioned on the information revealed so far, the probability that $Z_{i} = |S_{i}^{R}(v)| - z_{i}$ is at most
\begin{eqnarray*}
	\max_{z}\prob(\text{Bin}(s_{i}, r/n) = z) &\leq& \min\left\{1, O\left(\sqrt{\frac{n}{rs_{i}}}\right) \right\}\\
	&=&O\left(\min\left\{1, \sqrt{\frac{n}{rn^{ix}}}\right\}\right)\\
	&=&O\left(\min\left\{1, n^{\frac{1-ix-y}{2}} \right\} \right)\\
	&=&O\left(n^{-\frac{1}{2}\max\{ ix+y-1, 0 \} }\right),
\end{eqnarray*}
where we use \cref{lem:binomial} for the inequality in the first line. Suppose $x \in \left(\frac{1}{k+1}, \frac{1}{k} \right)$ for some integer $k \ge 8$. In this case, we have $i^{*} = \lfloor 1/x \rfloor = k$, and from the above we inductively have the bound
\begin{eqnarray*}
	\prob\left(\bm{m}_{R}(v) = \bm{m}_{R}(w)\,|\,\mathcal{V}\right) &\leq& \prob\left(\bigcap_{i=0}^{i^{*}}\left\{|S_{i}^{R}(v)| = |S_{i}^{R}(w)|\right\}\,|\,\mathcal{V}\right)\\
	&\leq& \prod_{i=0}^{i^{*}}O\left(n^{-\frac{1}{2}\max\{ ix+y-1, 0 \} }\right)\\
	&=&O\left(n^{-\frac{1}{2}\sum_{i=0}^{\lfloor 1/x \rfloor}\max\{ix+y-1,0\} }\right)\\
	&=&n^{-\frac{1}{2}f_{x}(y)+ O(\log^{-1}n)},
\end{eqnarray*}
which establishes \eqref{eq:to_show} for this case. 

Suppose now that $x = \frac{1}{k}$ for some $k \ge 8$. For this case $i^{*} = \lfloor 1/x \rfloor -1 = k-1$ and $c = \frac{d^{i^{*}+1}}{n} = n^{O(\log^{-1}n)}= \Theta(1)$. We run stage $i^{*}+1$ in the same way as stages $0$ through $i^{*}$. Let $s_{i^{*}+1} = |S_{i^{*}+1}(w) \setminus N_{i^{*}+1}(v)|$, and observe by \cref{cor:expansion_cor} that $s_{i^{*}+1} = (1+o(1))e^{-c}(1-e^{-c})n = \Theta(n)$. Using \cref{lem:binomial}, we can then upper bound probability that $|S_{i^{*}+1}^{R}(v)| = |S_{i^{*}+1}^{R}(w)|$ (conditioned on the history) by
$$
	\max_{z}\prob(\text{Bin}(s_{i^{*}+1}, r/n) = z) \leq \min\left\{1, O\left(\sqrt{\frac{n}{rs_{i^{*}+1}}}\right) \right\} = O(r^{-1/2}) = n^{-y/2 + O(\log^{-1}n)}.
$$
We then have the bound
\begin{eqnarray*}
	\prob\left(\bm{m}_{R}(v) = \bm{m}_{R}(w)\,|\,\mathcal{V}' \right) &\leq& n^{-\frac{1}{2}\sum_{i=0}^{i^{*}}\max\{ix + y - 1,0\} + O(\log^{-1}n)} \cdot n^{-\frac{1}{2}y + O(\log^{-1}n)}\\
	&=&n^{-\frac{1}{2}\sum_{i=0}^{\lfloor 1/x \rfloor -1}\max\{ix + y - 1,0\} - \frac{1}{2}y + O(\log^{-1}n)}.
\end{eqnarray*}
Note that
\begin{eqnarray*}
	\sum_{i=0}^{\lfloor1/x\rfloor - 1}\max\{ix+y-1,0\} &=& f_{x}(y) - \max\{\lfloor1/x\rfloor x + y-1,0\}\\
	&=& f_{x}(y) - \max\{k \cdot \frac{1}{k} + y - 1,0\}\\
	&=&f_{x}(y) - y.
\end{eqnarray*}
From this it follows that
$$
	\prob\left(\bm{m}_{R}(v) = \bm{m}_{R}(w)\,|\,\mathcal{V}' \right) \leq n^{-\frac{1}{2}f_{x}(y) + O(\log^{-1}n)},
$$
which establishes \eqref{eq:to_show} for this case. The proof of the theorem is finished.
\end{proof}

\section{Lower bound}\label{sec:lower_bound}

To illustrate the proof of the lower bound, let $x \in \left(\frac{1}{k+1}, \frac{1}{k}\right]$ for some $k \geq 2$ and $d = (n-1)p = n^{x + O(\log^{-1}n)}$. We will make use of the following result on the diameter of $\Gnp$ from~\cite{bollobas1998random}:

\begin{lemma}[\protect{\cite[Corollary 10.12]{bollobas1998random}}]\label{lem:diameter}
	Let $d = (n-1)p$, and suppose that a sequence $i = i(n) \geq 3$ satisfies $\frac{\log n}{i} - 3\log\log n \to \infty.$ If 
	$$
		\frac{d^{i-1}}{n} - 2\log n \to -\infty \quad\text{ and }\quad \frac{d^{i}}{n} - 2\log n \to \infty,
	$$
	then the diameter of $\Gnp$ equals $i$ w.h.p.
\end{lemma}

For our choice of $x$, by \cref{lem:diameter} the diameter of $\Gnp$ is $k+1$ w.h.p. (Indeed, this follows immediately after setting $i = k+1 \geq 3$ in the lemma.) This means that for any set $R$ and any vertex $v$, the multiset signature $\bm{m}_{R}(v)$ of $v$ with respect to $R$ is w.h.p.\ fully determined by the first $k+1$ coordinates, i.e., by $|S_{i}^{R}(v)|$ for $i \in \{ 0,1,\dots,k\}$. Using a simple counting argument, we will show that w.h.p.\ in $\Gnp$, for any set $R$, at least $(1+o(1))\frac{n}{2}$ vertices have the property that $|S_{i}^{R}(v)| \leq \max\left\{1, O\left(\frac{d^{i}|R|}{n}\right)\right\}$ for all $i  \in \{ 0,1,\dots,k\} $. For these \textit{typical} vertices, when $|R| \leq n^{y}$, the number of possible signatures with respect to $R$ is then crudely at most
$$
	\prod_{i=0}^{k} \max\left\{1, O\left(\frac{d^{i}|R|}{n}\right)\right\} = O\left(\prod_{i=0}^{k} n^{\max\{ix + y - 1, 0\}} \right) = n^{f_{x}(y) + O(\log^{-1} n)}.
$$
If the righthand side above is less than, say, $0.49n$, then we may conclude that there must be some pair of typical vertices $v,w$ with $\bm{m}_{R}(v) = \bm{m}_{R}(w)$. Recall that by \cref{lem:y1_and_y4} there exists a solution $y_{1} = y_{1}(x) \in (0,1]$ to the equation $f_{x}(y) = 1$ when $0 < x \leq 1/2.$ Let $y = y_{1} -\frac{C}{\log n}$ for a constant $C > 0$. By \cref{lem:exponent_func2} we have $f_{x}(y) \leq f_{x}(y_{1}) - \frac{C}{\log n} = 1 - \frac{C}{\log n}$. Thus if $C$ is chosen sufficiently large, then $n^{f_{x}(y) + O(\log^{-1}n)}$ can be made less than $0.49n$, which yields the desired result. 

\begin{theorem}
	Let $x \in \left(0, \frac{1}{2}\right]$ be fixed and let $d = (n-1)p = n^{x + O(\log^{-1}n)}$. W.h.p., $\Gnp$ has no multiset resolving set of size less than $n^{y_{1} + O(\log^{-1}n)}$, where $y_{1} = y_{1}(x) \in (0,1)$ is the unique solution to $f_{x}(y) = 1$.
\end{theorem}

Note that the diameter of $\Gnp$ is $2$ for $x  > \frac{1}{2}$, in which case there is no multiset resolving set for $\Gnp$ by a result from \cite{simanjuntak2017multiset}.

\begin{proof}
	Condition on the event $\mathcal{V}$. Suppose that $x \in \left(\frac{1}{k+1}, \frac{1}{k} \right]$ for some positive integer $k \geq 2$. Then, by \cref{lem:diameter}, the diameter of $\Gnp$ is $k+1$ w.h.p., and we also condition on this. Now, let $R \subseteq V$ have size $|R| = r$. For $i \in \{ 0,1,\dots,k\}$, we say that a vertex $v$ is $i$-\textit{atypical} with respect to $R$ if
$$
	|N_{i}^{R}(v)| \geq \max\left\{2(k+1)\frac{|N_{i}(v)|r}{n},\, 1\right\}
$$
and is $i$-\textit{typical} with respect to $R$ otherwise. If $v$ is $i$-typical for all $i \in \{ 0,1,\dots,k \}$ with respect to $R$, we simply say $v$ is \emph{typical} with respect to $R$. Our goal is to show that there are at least $(1+o(1))n/2$ typical vertices.

Let $A_{i}^{R}$ be the set of vertices which are $i$-atypical with respect to $R$, and let
$$
	P_{i}^{R} = \{(v,w)\,:\,v \in R,\,w \in A_{i}^{R},\,d(v,w) \leq i\}.
$$
On one hand, for any $i \in \{ 0,1,\dots,k \}$, we have
$$
	|P_{i}^{R}| \geq 2(k+1)\frac{\min_{w \in A_{i}^{R}}|N_{i}(w)|r}{n}|A_{i}^{R}|,
$$
since there are at least $2(k+1)\frac{\min_{w \in A_{i}^{R}}|N_{i}(w)|r}{n}$ ``partners'' in $R$ for a fixed $w \in A_{i}^{R}$. On the other hand,
$$
	|P_{i}^{R}| \leq \max_{v \in R}|N_{i}(v)|r,
$$
since there are at most $\max_{v \in R}|N_{i}(v)|$ partners in $A_{i}^{R}$ for a fixed $v \in R$. Thus,
$$
	|A_{i}^{R}| \leq \frac{n}{2(k+1)}\cdot\frac{\max_{w \in A_{i}^{R}}|N_{i}(w)|}{\min_{v \in R}|N_{i}(v)|} = (1+o(1))\frac{n}{2(k+1)}
$$
where the last equality holds from the fact that $|N_{i}(w)| = (1+o(1))|N_{i}(v)|$ for all pairs $v$ and $w$ and all $i \in \{ 0,1,\dots,k \}$ by \cref{cor:expansion_cor}~(i). It follows that the number of vertices which are $i$-atypical for some $i \in \{ 0,1,\dots,k \}$ is at most $(1+o(1))\frac{n}{2}$, i.e., there are at least $(1+o(1))\frac{n}{2}$ typical vertices. Note that this holds for \textit{any} subset $R$.

Now, since the diameter is $k+1$, for any $R \subseteq V$ and $v \in V$, the signature $\bm{m}_{R}(v)$ is fully determined by the coordinates $0,1,\dots,k$. First suppose that $\frac{1}{k+1} < x < \frac{1}{k}$. Note that in this case we have $i^{*} = k$. Let $\epsilon = \frac{C}{\log n}$ with $C>0$ a constant to be determined, and for a contradiction suppose that there exists a multiset resolving set $R$ of size $n^{y}$ with $y \leq y_{1} - \epsilon$. By \cref{cor:expansion_cor}~(i), for any $v \in V$ and any $i \in \{ 0,1,\dots,k \}$, we have $|N_{i}(v)| = (1+o(1))d^{i} = \Theta(n^{ix}).$ Thus, if $v$ is typical with respect to $R$, then for any $i \in \{ 0,1,\dots,k \}$, there are at most
\begin{equation}\label{eq:typical_coordinate}
	\max\left\{(1+o(1))2(k+1)\frac{d^{i}r}{n},\,1 \right\} = O\left(\max\{n^{ix +y - 1}, 1\} \right) =O\left(n^{\max\{ix+y-1,0\}} \right)
\end{equation}
possibilities for the $i$th coordinate of $\bm{m}_{R}(v)$. Therefore, the number of signatures which are allowed for typical vertices is at most
$$
	O\left(\prod_{i=0}^{k}n^{\max\{ix+y-1,0\}}\right)= n^{f_{x}(y) + O(\log^{-1}n)} \leq n^{f_{x}(y_{1}) - \epsilon + O(\log^{-1}n)},
$$
where in the final inequality we use that $f_{x}(y) \leq f_{x}(y_{1} - \epsilon) \leq f_{x}(y_{1}) - \epsilon$ by Lemmas~\ref{lem:exponent_func1} and~\ref{lem:exponent_func2}. (Note that $f_{x}(y_{1} - \epsilon) > 0$ holds for $n$ sufficiently large since $f_{x}$ is continuous by \cref{lem:exponent_func1} and $\epsilon = o(1)$, so we can indeed apply \cref{lem:exponent_func2}.) If $C$ is chosen large enough, then the right-hand-side above can be made smaller than, say, $0.49n$. By our previous arguments, there are at least $(1+o(1))\frac{n}{2}$ typical vertices with respect to $R$. But since there are only $0.49n$ possible signatures for these vertices, there must be a pair $v,w$ of typical vertices with $\bm{m}_{R}(v) = \bm{m}_{R}(w)$, which gives us the desired contradiction.

The argument above extends easily to the case that $x = \frac{1}{k}$ for some positive integer $k \ge 2$. In this case we have $i^{*} = k-1$, so the bound \eqref{eq:typical_coordinate} holds for any $i \in \{ 0,1,\dots,k-1 \}$ in this case. Trivially, we have that $|N_{k}(v)| = O(n)$ for any $v$, which means there are at most $O\left(n\cdot\frac{r}{n}\right) = O(r) = O(n^{y})$ possibilities for the $k$th coordinate of any typical vertex's signature with respect to $R$. There are then at most 
$$
	O\left( n^{y}\cdot \prod_{i=0}^{k-1}n^{\max\{ix+y-1,0\}} \right) = n^{f_{x}(y) + O(\log^{-1}n)}
$$
possible signatures for typical vertices, and the rest of the argument follows as before.
\end{proof}

\section{Conclusion}\label{sec:conclusion}

We have shown that for the regime $d = (n-1)p = n^{x + O(\log^{-1}n)}$ for fixed $0<x \leq 1/8$, the muiltiset metric dimension $\beta_{\text{ms}}(\Gnp)$ of $\Gnp$ satisfies $\beta_{\text{ms}}(\Gnp) \leq n^{y_{4} + O(\log^{-1}n)}$ w.h.p., and for fixed $0 < x \leq \frac{1}{2}$, $\beta_{\text{ms}}(\Gnp)  \geq n^{y_{1} + O(\log^{-1}n)}$ w.h.p., where $y_{4}$ and $y_{1}$ are constants in $(0,1)$ which are explicitly computable from $x$. We chose the scaling $d = n^{x +O(\log^{-1}n)}$ primarily to make our results cleaner to state. Since our arguments depend only on typical expansion properties of $\Gnp$, which hold in a much more general regime than the one we consider, we believe it is worth investigating whether extensions are possible. 

We propose two directions for future research. The first is to determine whether $\beta_{\text{ms}}(\Gnp)$ is finite w.h.p.\ for $d = n^{x}$, $\frac{1}{8} < x \leq \frac{1}{2}$. The second is to tighten the gap between the lower and upper bounds for $\beta_{\text{ms}}(\Gnp)$ for $x \in (0,\frac{1}{2})$. New proof strategies are likely needed for both problems.

\bibliography{multiset_dimension_refs}

\begin{thebibliography}{10}

\bibitem{alfarisi2020note}
Ridho Alfarisi, Yuqing Lin, Joe Ryan, Dafik Dafik, and Ika~Hesti Agustin.
\newblock A note on multiset dimension and local multiset dimension of graphs.
\newblock {\em Statistics, Optimization \& Information Computing},
  8(4):890--901, 2020.

\bibitem{bollobas1998random}
B{\'e}la Bollob{\'a}s.
\newblock {\em Random graphs}.
\newblock Springer, 1998.

\bibitem{bollobas2013metric}
B{\'e}la Bollob{\'a}s, Dieter Mitsche, and Pawe\l{} Pra\l{}at.
\newblock Metric dimension for random graphs.
\newblock {\em The Electronic Journal of Combinatorics}, 20(4):P52, 2013.

\bibitem{bong2021some}
Novi~H Bong and Yuqing Lin.
\newblock Some properties of the multiset dimension of graphs.
\newblock {\em Electron. J. Graph Theory Appl.}, 9(1):215--221, 2021.

\bibitem{diaz2025metric}
Josep D{\'\i}az, Harrison Hartle, and Cristopher Moore.
\newblock The metric dimension of sparse random graphs.
\newblock {\em The Electronic Journal of Combinatorics}, 33(1):P1.59, 2026.

\bibitem{dudek2022localization}
Andrzej Dudek, Sean English, Alan Frieze, Calum MacRury, and Pawe{\l}
  Pra{\l}at.
\newblock Localization game for random graphs.
\newblock {\em Discrete Applied Mathematics}, 309:202--214, 2022.

\bibitem{dudek2019note}
Andrzej Dudek, Alan Frieze, and Wesley Pegden.
\newblock A note on the localization number of random graphs: diameter two
  case.
\newblock {\em Discrete Applied Mathematics}, 254:107--112, 2019.

\bibitem{feller1957introduction}
William Feller.
\newblock {\em An introduction to probability theory and its applications},
  volume~1.
\newblock Wiley Series in Probability and Mathematical Statistics, 1957.

\bibitem{frieze2016introduction}
Alan Frieze and Micha{\l} Karo{\'n}ski.
\newblock {\em Introduction to random graphs}.
\newblock Cambridge University Press, 2016.

\bibitem{gil2019distance}
Reynaldo Gil-Pons, Yunior Ram{\'\i}rez-Cruz, Rolando Trujillo-Rasua, and
  Ismael~G Yero.
\newblock Distance-based vertex identification in graphs: The outer multiset
  dimension.
\newblock {\em Applied Mathematics and Computation}, 363:124612, 2019.

\bibitem{hakanen2024complexity}
Anni Hakanen and Ismael~G Yero.
\newblock Complexity and equivalency of multiset dimension and id-colorings.
\newblock {\em Fundamenta Informaticae}, 191(3-4):315--330, 2024.

\bibitem{harary1976metric}
F.~Harary and R.A. Melter.
\newblock The metric dimension of a graph.
\newblock {\em Ars Combinatoria}, 2:191--195, 1976.

\bibitem{JLR}
Svante Janson, Tomasz \L{}uczak, and Andrzej Ruci\'nski.
\newblock {\em Random graphs}.
\newblock John Wiley \& Sons, 2011.

\bibitem{klavvzar2023further}
Sandi Klav{\v{z}}ar, Dorota Kuziak, and Ismael~G Yero.
\newblock Further contributions on the outer multiset dimension of graphs.
\newblock {\em Results in Mathematics}, 78(2):50, 2023.

\bibitem{lichev2023localization}
Lyuben Lichev, Dieter Mitsche, and Pawe{\l} Pra{\l}at.
\newblock Localization game for random geometric graphs.
\newblock {\em European Journal of Combinatorics}, 108:103616, 2023.

\bibitem{mitsche2015limiting}
Dieter Mitsche and Juanjo Ru{\'e}.
\newblock On the limiting distribution of the metric dimension for random
  forests.
\newblock {\em European Journal of Combinatorics}, 49:68--89, 2015.

\bibitem{odor2022role}
Gergely Odor.
\newblock The role of adaptivity in source identification with time queries.
\newblock Technical report, EPFL, 2022.

\bibitem{odor2021sequential}
Gergely Odor and Patrick Thiran.
\newblock Sequential metric dimension for random graphs.
\newblock {\em Journal of Applied Probability}, 58(4):909--951, 2021.

\bibitem{simanjuntak2017multiset}
Rinovia Simanjuntak, Presli Siagian, and Tomas Vetrik.
\newblock The multiset dimension of graphs.
\newblock {\em arXiv preprint arXiv:1711.00225}, 2017.

\bibitem{slater1975leaves}
P.~Slater.
\newblock Leaves of trees.
\newblock {\em Congressus Numerantium}, 14:549--559, 1975.

\bibitem{tillquist2023getting}
Richard~C Tillquist, Rafael~M Frongillo, and Manuel~E Lladser.
\newblock Getting the lay of the land in discrete space: A survey of metric
  dimension and its applications.
\newblock {\em SIAM Review}, 65(4):919--962, 2023.

\bibitem{tillquist2019low}
Richard~C Tillquist and Manuel~E Lladser.
\newblock Low-dimensional representation of genomic sequences.
\newblock {\em Journal of mathematical biology}, 79(1):1--29, 2019.

\bibitem{tillquist2019multilateration}
Richard~D Tillquist and Manuel~E Lladser.
\newblock Multilateration of random networks with community structure.
\newblock {\em arXiv preprint arXiv:1911.01521}, 2019.

\end{thebibliography}

\end{document}